\documentclass[11pt,twoside]{article}
\usepackage[T1]{fontenc}
\usepackage{textcomp}
\usepackage{mathrsfs}
\usepackage{amsmath}
\usepackage{amssymb}
\usepackage{fancyhdr}
\usepackage{latexsym}
\usepackage{bbding}
\usepackage{mathrsfs}
\usepackage{wasysym}
\usepackage{cite}
\usepackage{multicol,graphics}
\usepackage{xcolor}

\RequirePackage{times}

\def\XXint#1#2#3{{\setbox0=\hbox{$#1{#2#3}{\int}$ }
		\vcenter{\hbox{$#2#3$ }}\kern-.6\wd0}}

\newtheorem{theorem}{Theorem}[section]
\newtheorem{lemma}[theorem]{Lemma}

\newtheorem{definition}[theorem]{Definition}

\newtheorem{proposition}[theorem]{Proposition}
\numberwithin{equation}{section}
\newenvironment{proof}[1][Proof]{\noindent\textbf{#1.} }{\hfill $\Box$}

\allowdisplaybreaks[4]
\makeatletter
\begin{document}
	
\title{Global well-posedness of the Navier--Stokes equations and the Keller--Segel system in variable Fourier--Besov spaces\footnote{E-mails: gvh@westlake.edu.cn, jihzhao@163.com.}}
\author{Gast\'on Vergara-Hermosilla$^{\text{1}}$, Jihong Zhao$^{\text{2}}$\\
[0.2cm] {\small $^{\text{1}}$ Institute for Theoretical Sciences, Westlake University, Hangzhou, Zhejiang 310024, China.}\\
[0.2cm] {\small $^{\text{2}}$School of Mathematics and Information Science, Baoji University of Arts and Sciences,}\\
[0.2cm] {\small  Baoji, Shaanxi 721013,  China}}

\date{\today}
\maketitle
\begin{abstract}
In this paper, we study the Cauchy problem of the classical incompressible Navier--Stokes equations and the parabolic-elliptic Keller--Segel system in the framework of the Fourier--Besov spaces with variable regularity and integrability indices. By fully using some basic properties of these variable function spaces, we establish the linear estimates in variable Fourier--Besov spaces for the heat equation. Such estimates are fundamental for solving certain PDE's of parabolic type. As an applications, we prove global well-posedness in variable Fourier--Besov spaces for the 3D classical incompressible Navier--Stokes equations and the 3D parabolic-elliptic Keller--Segel system.
\end{abstract}
\smallbreak

\textit{Keywords}: Variable Fourier--Besov spaces; Navier--Stokes equations; Keller--Segel system; global well-posedness
\smallskip

\textit{2020 AMS Subject Classification}:   35A01,  35Q30, 46E30,  92C17

%%%%%%%%%%%%%%%%%%%%%%%%%%%%%%%%%%%%%%%%%%%%%	
\section{Introduction}

 \smallbreak

In this paper, we are concerned with the analysis of systems of partial differential equations of parabolic type, by considering as functional framework some
 Fourier--Besov spaces with variable regularity and integrability indices. More precisely, by applying some of the main properties of these functional spaces, we are able to establish global well-posedness theorems for the classical  3D incompressible Navier--Stokes equations, in the context fluid dynamics, and the 3D elliptic-parabolic Keller--Segel system, in the context of mathematical biology. In the following, we motivate and present these main results.

\subsection{The 3D incompressible Navier--Stokes equations}
The 3D incompressible Navier--Stokes equations read
\begin{equation} \label{eq1.1}
\left\{
\begin{aligned}
& \partial_t u -\Delta u + u\cdot\nabla u +\nabla \pi =f,  \\
& \nabla\cdot u=0,\\
& u(0,x) =u_0(x).
\end{aligned}
\right.
\end{equation}
Here $u$ is a vector field representing the velocity of the fluid,  $\pi$ is a scalar denoting the pressure, $u_{0}$ is a given initial velocity with compatibility condition $\nabla\cdot u_{0}=0$ and $f$ is a given external force.
\smallbreak

It is well-known that the Navier--Stokes equations \eqref{eq1.1} has a global smooth solution if the initial data is small enough in the scale invariant spaces. Let us recall that the scaling of the Navier--Stokes equations \eqref{eq1.1} is following: if $u$ solves \eqref{eq1.1} with initial data $u_{0}$ and $f$, so does for any $\lambda>0$ the vector field $u_{\lambda}:(t,x)\mapsto\lambda u(\lambda^{2}t,\lambda x)$ with initial data $u_{0, \lambda}$ and $f_{\lambda}$, where
\begin{equation*}
 u_{0, \lambda}: x\mapsto\lambda u_{0}(\lambda x) \ \ \ \text{and}\ \ \  f_{\lambda}:(t,x)\mapsto\lambda^{3} f(\lambda^{2}t,\lambda x).
\end{equation*}
Clearly,  the following function spaces are scale invariant for initial data:
$$
  \dot{H}^{\frac{1}{2}}(\mathbb{R}^{3})\subseteq L^{3}(\mathbb{R}^{3})\subseteq \dot{B}^{-1+\frac{3}{p}}_{p,q}(\mathbb{R}^{3})\subseteq BMO^{-1} \subseteq\dot{B}^{-1}_{\infty,\infty} (\mathbb{R}^{3}), \ \ 3\leq p\leq\infty, \ 1\leq q\leq 2.
$$
 The well-posed result in $\dot{H}^{\frac{1}{2}}(\mathbb{R}^{3})$ is due to  Fujita--Kato \cite{FK64} (see also \cite{L34}, where the smallness of $u_{0}$ is measured by $\|u_{0}\|_{L^2}\|\nabla u_{0}\|_{L^2}$). Since then, a number of works have been devoted to establishing similar well-posed results for larger class of initial data. Kato \cite{K84} proved the global well-posedness for small
data and local well-posedness for large data in the critical Lebesgue space $L^{3}(\mathbb{R}^{3})$. Later Cannone \cite{C97} and Planchon \cite{P96} proved the same results for initial data in the critical Besov spaces $\dot{B}^{-1+\frac{3}{p}}_{p,\infty}(\mathbb{R}^{3})$ (see also the different proof by Chemin \cite{C92}), and Koch--Tataru \cite{KT01} for initial data in $BMO^{-1}$. For the 3D largest scaling invariant space $\dot{B}^{-1}_{\infty,\infty}(\mathbb{R}^{3})$,  Bourgain and Pavlovi\'{c} \cite{BP08} proved ill-posedness for the Navier--Stokes equations \eqref{eq1.1} in $\dot{B}^{-1}_{\infty,\infty}(\mathbb{R}^{3})$. We refer to Yoneda \cite{Y10} and Wang \cite{W15} for more ill-posedness results.
\smallbreak

The first goal of this paper is to establish the global well-posedness of the Navier--Stokes equations \eqref{eq1.1} in the framework of Fourier--Besov spaces $F\dot{B}^{s(\cdot)}_{p(\cdot),1}(\mathbb{R}^{3})$ with variable regularity and integrability indices. Notice that the Fourier--Besov spaces $F\dot{B}^{s}_{p,r}(\mathbb{R}^{n})$ with constant regularity and integrability indices was initially introduced by Iwabuchi \cite{I11} to study the well-posed and ill-posed issues of the parabolic-elliptic Keller--Segel system. Later on,  Iwabuchi--Takada \cite{IT14} used the critical $F\dot{B}^{s}_{p,r}$-spaces to obtain a global well-posedness class (uniformly with respect to
the angular velocity $\Omega$) for the Navier--Stokes--Coriolis system. Konieczny--Yoneda \cite{KY11} showed global well--posedness and
asymptotic stability of small solutions for the 3D Navier--Stokes equations in the critical Fourier--Besov spaces. On the other hand, some progresses of well-posed issues of  the Navier--Stokes equations \eqref{eq1.1} in the variable function spaces have been obtained recently.
By overcoming the difficulties caused by the boundedness of the Riesz potential in a variable Lebesgue spaces, Chamorro and the first author of this paper \cite{CV24} established the global well-posedness of the Navier--Stokes equations \eqref{eq1.1} in the framework of variable Lebesgue spaces $L^{p(\cdot)}(\mathbb{R}^{3})$.
Ru \cite{R18} proved the global well-posedness of the  Navier--Stokes equations \eqref{eq1.1} with small initial data in the Fourier--Besov spaces $F\dot{B}^{s(\cdot)}_{p(\cdot),r}(\mathbb{R}^{3})$.
We refer to \cite{V241, V242} for more applications of variable function spaces in the field of partial differential equations.
\smallbreak

Now we state our main result for the 3D incompressible Navier--Stokes equations \eqref{eq1.1}.

\begin{theorem}\label{th1.1}
Let $p(\cdot)\in \mathcal{P}^{\log}_{0}(\mathbb{R}^{3})$ with $2\leq p^{-}\leq p(\cdot) \leq p^{+}\leq 6$, $1\leq \rho\leq+\infty$. For any $u_{0}\in F\dot{B}^{2-\frac{3}{p(\cdot)}}_{p(\cdot),1}(\mathbb{R}^{3})$ and  $f\in\mathcal{L}^{1}(0,\infty; F\dot{B}^{2-\frac{3}{p(\cdot)}}_{p(\cdot),1}(\mathbb{R}^{3}))\cap \mathcal{L}^{1}(0,\infty; F\dot{B}^{\frac{1}{2}}_{2,1}(\mathbb{R}^{3}))$, there exists a positive constant $\varepsilon$ such that if the initial data $u_{0}$ and $f$ satisfy
\begin{equation*}
\|u_{0}\|_{F\dot{B}^{2-\frac{3}{p(\cdot)}}_{p(\cdot),1}}+\|f\|_{\mathcal{L}^{1}_{t}(F\dot{B}^{2-\frac{3}{p(\cdot)}}_{p(\cdot),1})}+\|f\|_{\mathcal{L}^{1}_{t}(F\dot{B}^{\frac{1}{2}}_{2,1})}\leq \varepsilon,
\end{equation*}
then the Navier--Stokes equations \eqref{eq1.1} admits a unique global solution $u$ such that
\begin{equation*}
u\in \mathcal{L}^{\rho}(0,\infty; F\dot{B}^{2-\frac{3}{p(\cdot)}+\frac{2}{\rho}}_{p(\cdot),1}(\mathbb{R}^{3}))\cap \mathcal{L}^{1}(0,\infty; F\dot{B}^{\frac{5}{2}}_{2,1}(\mathbb{R}^{3}))\cap \mathcal{L}^{\infty}(0,\infty; F\dot{B}^{\frac{1}{2}}_{2,1}(\mathbb{R}^{3})).
\end{equation*}
\end{theorem}

Let us remark that, the techniques we used to prove Theorem \ref{th1.1} is different from \cite{R18}, we shall employ the Fourier--Besov spaces with constant exponents, rather than Besov type spaces used in \cite{R18}, to be the auxiliary spaces to find global solutions of the equations \eqref{eq1.1}. In our view, choosing the Fourier--Besov spaces with constant exponents as auxiliary spaces seems more natural way to find solutions of PDEs in the framework of variable Fourier--Besov spaces.

\subsection{The elliptic-parabolic Keller--Segel system}

The 3D parabolic-elliptic Keller--Segel system write
\begin{equation}\label{eq1.2}
\begin{cases}
  \partial_{t}u-\Delta u+\nabla\cdot(u\nabla v)=f,\\
    -\Delta v=u,\\
  u(0,x)=u_0(x).
\end{cases}
\end{equation}
Here $u$ and $v$ stand for the cell density and the concentration of the chemical attractant, respectively, and both are unknown functions of the time variable $t\in\mathbb{R}^{+}$  and  of the space variable $x\in \mathbb{R}^{3}$,  $u_{0}$ is a given initial cell density and $f$ is a given external force.
\smallbreak

The above system \eqref{eq1.2} is a simplified version of the mathematical model of chemotaxis, which was first introduced by Keller--Segel \cite{KS70} to describe the collective movement of bacteria possibly leading to cell aggregation by chemotactic
effect.  The Keller--Segel system \eqref{eq1.2} also satisfies the following scale invariant property: if $u$ solves \eqref{eq1.2} with initial data $u_{0}$ and $f$, so does for any $\lambda>0$ the cell density $u_{\lambda}:(t,x)\mapsto\lambda^{2} u(\lambda^{2}t,\lambda x)$ with initial data $u_{0, \lambda}$ and $f_{\lambda}$, where
\begin{equation*}
 u_{0, \lambda}: x\mapsto\lambda^{2} u_{0}(\lambda x) \ \ \ \text{and}\ \ \  f_{\lambda}:(t,x)\mapsto\lambda^{4} f(\lambda^{2}t,\lambda x).
\end{equation*}
It is clear that the following function spaces are scaling invariant for initial data:
$$
  \dot{H}^{-\frac{1}{2}}(\mathbb{R}^{3})\subseteq L^{\frac{3}{2}}(\mathbb{R}^{3})\subseteq \dot{B}^{-2+\frac{3}{p}}_{p,q}(\mathbb{R}^{3})\subseteq BMO^{-2} \subseteq\dot{B}^{-2}_{\infty,\infty} (\mathbb{R}^{3}), \ \ 3\leq p\leq\infty, \ 1\leq q\leq 2.
$$
Similar to the Navier--Stokes equations \eqref{eq1.1}, a number of works have been devoted to establishing global well-posedness of the Keller--Segel system \eqref{eq1.2} with initial data in the above scale invariant spaces.  Corrias--Perthame--Zaag \cite{CPZ04} studied global solutions for initial data $ u_{0}\in L^{1} (\mathbb{R}^{3})\cap L^{\frac{3}{2}}(\mathbb{R}^{3})$ with
$\| u_{0}\|_{L^{\frac{3}{2}}}$ being sufficiently small, see also Kozono--Sugiyama \cite{KS08} for 2D case.
 The well-posed result in Besov spaces $\dot{B}^{-2+\frac{3}{p}}_{p,q}(\mathbb{R}^{3})$ with $1\leq p<\infty$ and $1\leq q\leq\infty$ is due to Iwabuchi \cite{I11}, where the ill-posedness is also established in critical Besov spaces $\dot{\mathcal{B}}^{-2}_{\infty,q}(\mathbb{R}^{3})$ with $2<q\leq\infty$.  Iwabuchi--Nakamura \cite{IN13} finally showed global existence of solutions for small initial data in $BMO^{-2}$ through the Triebel--Lizorkin spaces, which combining the well-posed results in \cite{I11},  we know that, the simplified Keller--Segel system \eqref{eq1.2} is well-posed in $\dot{B}^{-2+\frac{3}{p}}_{p,\infty}(\mathbb{R}^{3})$ and $\dot{F}^{-2}_{\infty,2}(\mathbb{R}^{3})$ with $\max\{1, \frac{3}{2}\}<p<\infty$, but ill-posed in $\dot{\mathcal{B}}^{-2}_{\infty,q}(\mathbb{R}^{3})$ with $2<q\leq\infty$. We refer the readers to \cite{LW23, LYZ23, NY20, NY22, XF22} for more ill-posedness results.
 \smallbreak

Notice that, in\cite{VZ24}, the authors of this paper have established the global well-posedness of the Keller--Segel system \eqref{eq1.2} in the framework of variable Lebesgue spaces. Indeed,  based on carefully examining the algebraical structure of the Keller--Segel system \eqref{eq1.2}, we can reduce system \eqref{eq1.2} into the generalized nonlinear heat equation to overcome the difficulties caused by the boundedness of the Riesz potential in a variable Lebesgue spaces, then by mixing some structural properties of the variable Lebesgue spaces with the optimal decay estimates of the heat kernel, we obtain global existence of solutions of the Keller--Segel system \eqref{eq1.2} in the variable Lebesgue spaces.
 \smallbreak

The second  goal of this paper is to prove the global well-posedness of the 3D Keller--Segel system \eqref{eq1.2} in the framework of  variable Fourier--Besov spaces. The main result is as follows.
\begin{theorem}\label{th1.2}
Let $p(\cdot)\in \mathcal{P}^{\log}_{0}(\mathbb{R}^{3})$ with $2\leq p^{-}\leq p(\cdot) \leq p^{+}\leq 6$, $1\leq \rho\leq+\infty$. For any $u_{0}\in F\dot{B}^{1-\frac{3}{p(\cdot)}}_{p(\cdot),1}(\mathbb{R}^{3})$ and $f\in\mathcal{L}^{1}(0,\infty; F\dot{B}^{1-\frac{3}{p(\cdot)}}_{p(\cdot),1}(\mathbb{R}^{3}))\cap \mathcal{L}^{1}(0,\infty; F\dot{B}^{-\frac{1}{2}}_{2,1}(\mathbb{R}^{3}))$, there exists a positive constant $\varepsilon$ such that if the initial data $u_{0}$ and $f$ satisfy
\begin{equation*}
\|u_{0}\|_{F\dot{B}^{1-\frac{3}{p(\cdot)}}_{p(\cdot),1}}+\|f\|_{\mathcal{L}^{1}_{t}(F\dot{B}^{1-\frac{3}{p(\cdot)}}_{p(\cdot),1})}+\|f\|_{\mathcal{L}^{1}_{t}(F\dot{B}^{-\frac{1}{2}}_{2,1})}\leq \varepsilon,
\end{equation*}
then the Keller--Segel system \eqref{eq1.2} admits a unique global solution $u$ such that
\begin{equation*}
u\in \mathcal{L}^{\rho}(0,\infty; F\dot{B}^{1-\frac{3}{p(\cdot)}+\frac{2}{\rho}}_{p(\cdot),1}(\mathbb{R}^{3}))\cap \mathcal{L}^{1}(0,\infty; F\dot{B}^{\frac{3}{2}}_{2,1}(\mathbb{R}^{3}))\cap \mathcal{L}^{\infty}(0,\infty; F\dot{B}^{-\frac{1}{2}}_{2,1}(\mathbb{R}^{3})).
\end{equation*}
\end{theorem}
 \smallbreak

The structure of this paper is arranged as follows. In Section 2, we introduce some conventions and notations, and state some basic results of the (variable) Fourier--Besov spaces. In Section 3, we establish the linear estimates of heat equation in the framework of the Fourier--Besov spaces with variable regularity and integrability indices. In Section 4, we present the proof of Theorem \ref{th1.1}, while the proof of Theorem \ref{th1.2} will be given in Section 5 .

\section{Preliminaries}
In this section, we introduce some conventions and notations, and state some basic results of variable function spaces.  Let $\mathcal{S}(\mathbb{R}^n)$ be the Schwartz class of rapidly decreasing functions on $\mathbb{R}^n$, and $\mathcal{S}'(\mathbb{R}^n)$ the space of tempered
distributions.  Given  $f\in\mathcal{S}(\mathbb{R}^{n})$, the Fourier transform $\mathcal{F}(f)$ (or $\widehat{f}$) is defined by
$$
  \mathcal{F}(f)(\xi)=\widehat{f}(\xi):=\frac{1}{(2\pi)^{\frac{n}{2}}}\int_{\mathbb{R}^{n}}f(x)e^{-ix\cdot\xi}dx.
$$
More generally, the Fourier transform of  a tempered distribution $f\in\mathcal{S}'(\mathbb{R}^{n})$ is defined by the dual argument in the standard way. Throughout this paper,  the letters $C$ and $C_{i}$ ($i=1,2, \cdots$) stand for the generic harmless constants, whose meaning is clear from the context. For brevity, we shall use the notation
$u\lesssim v$ instead of $u\leq Cv$, and $u\approx v$ means that $u\lesssim v$
and $v\lesssim u$.   We denote $\{d_{j}\}_{j\in\mathbb{Z}}$ a generic element of
$l^{1}(\mathbb{Z})$ so that $d_{j}\ge0$ and
$\|d_{j}\|_{l^{1}(\mathbb{Z})}=1$.

 \subsection{Variable Lebesgue spaces}
Spaces of variable integrability, also known as variable Lebesgue spaces, are generalization of the classical Lebesgue spaces, replacing the constant $p$ with a variable exponent function $p(\cdot)$.  The origin of the variable Lebesgue spaces $L^{p(\cdot)}(\mathbb{R}^{n})$ can be traced back to Orlicz \cite{O31} in 1931, and then studied two decades later by Nakano \cite{N50, N51}, but the modern development started with the foundational paper of Kov\'{a}\v{c}ik--R\'{a}kosn\'{i}k \cite{KR91}, as well as the papers of Cruz-Uribe \cite{C03} and Diening \cite{D04}.  We begin with a fundamental definition.  Let $\mathcal{P}_{0}(\mathbb{R}^{n})$ be the set of all Lebesgue measurable functions $p(\cdot):\mathbb{R}^{n}\rightarrow[1,+\infty)$ such that
\begin{equation*}
1< p^{-}:=\operatorname{essinf}_{x\in\mathbb{R}^{n}}p(x), \ \  p^{+}:=\operatorname{esssup}_{x\in\mathbb{R}^{n}}p(x)<+\infty.
\end{equation*}
For any $p(\cdot)\in \mathcal{P}_{0}(\mathbb{R}^{n})$,  a measurable function $f$, we denote
\begin{equation}\label{eq2.1}
   \|f\|_{L^{p(\cdot)}}:=\inf\left\{\lambda>0:  \rho_{p(\cdot)}\left(\frac{f}{\lambda}\right)\leq 1\right\},
\end{equation}
where the modular function $\rho_{p(\cdot)}$ associated with $p(\cdot)$ is given by
\begin{equation*}
   \rho_{p(\cdot)}(f):=\int_{\mathbb{R}^{n}}|f(x)|^{p(x)}dx.
\end{equation*}
If the set on the right-hand side of  \eqref{eq2.1} is empty then we denote $\|f\|_{L^{p(\cdot)}}=\infty$. Now let us recall the definition of the Lebesgue spaces $L^{p(\cdot)}(\mathbb{R}^{n})$ with variable exponent.

\begin{definition}\label{de2.1}
Given $p(\cdot)\in \mathcal{P}_{0}(\mathbb{R}^{n})$, we define the variable Lebesgue spaces $L^{p(\cdot)}(\mathbb{R}^{n})$ to be the set of Lebesgue measurable functions $f$ such that   $\|f\|_{L^{p(\cdot)}}<+\infty$.
\end{definition}

The variable Lebesgue spaces $L^{p(\cdot)}(\mathbb{R}^{n})$ retain some good properties of the usual Lebesgue spaces, such as $L^{p(\cdot)}(\mathbb{R}^{n})$ is a Banach space associated with the norm $\|\cdot\|_{L^{p(\cdot)}}$, and we have the following  H\"{o}lder's inequality (cf. \cite{CF13}, Corollary 2.28; \cite{DHHR11}, Lemma 3.2.20). For a complete  presentation of the theory of a  variable Lebesgue spaces, we refer to the readers to see books \cite{CF13, DHHR11}.
\smallbreak

\begin{lemma}\label{le2.2}
Given two exponent functions $p_{1}(\cdot), p_{2}(\cdot)\in \mathcal{P}_{0}(\Omega)$, define $p(\cdot)\in \mathcal{P}_{0}(\Omega)$ by $\frac{1}{p(x)}=\frac{1}{p_{1}(x)}+\frac{1}{p_{2}(x)}$. Then there exists a constant $C$ such that for all $f\in L^{p_{1}(\cdot)}(\Omega)$ and $g\in L^{p_{2}(\cdot)}(\Omega)$, we have $fg\in L^{p(\cdot)}(\Omega)$ and
\begin{equation}\label{eq2.2}
   \|fg\|_{L^{p(\cdot)}}\leq C \|f\|_{L^{p_{1}(\cdot)}}\|g\|_{L^{p_{2}(\cdot)}}.
\end{equation}
\end{lemma}

However,  not all properties of the usual Lebesgue spaces $L^{p}(\mathbb{R}^{n})$ can be generalized to the variable Lebesgue spaces $L^{p(\cdot)}(\mathbb{R}^{n})$. For example, the variable Lebesgue spaces are not translation invariant, thus the convolution of two functions $f$ and $g$ is not well-adapted, and the Young's inequality are not valid anymore (cf. \cite{CF13}, Section 5.3). In consequence, new ideas and techniques are needed to tackle with the boundedness of many classical operators appeared in the mathematical analysis of PDEs. A classical approach to study these difficulties is to consider some constraints on the variable exponent, and the most common one is given by the so-called \textit{log-H\"{o}lder continuity condition}.
\begin{definition}\label{de2.3}
Let $p(\cdot)\in \mathcal{P}_{0}(\mathbb{R}^{n})$ such that  there exists a limit $\frac{1}{p_{\infty}}=\lim_{|x|\rightarrow\infty}\frac{1}{p(x)}$.
\begin{itemize}
\item We say that $p(\cdot)$ is locally log-H\"{o}lder continuous if for all $x,y\in \mathbb{R}^{n}$,  there exists a constant $C$ such that $\big{|}\frac{1}{p(x)}-\frac{1}{p(y)}\big{|}\leq \frac{C}{\log(e+\frac{1}{|x-y|})}$;
\item We say that $p(\cdot)$ satisfies the log-H\"{o}lder decay condition if for all $x\in \mathbb{R}^{n}$,  there exists a constant $C$ such that $\big{|}\frac{1}{p(x)}-\frac{1}{p_{\infty}}\big{|}\leq \frac{C}{\log(e+|x|)}$;
\item We say that  $p(\cdot)$ is globally log-H\"{o}lder continuous in $\mathbb{R}^n$ if it is locally \textit{log-H\"{o}lder continuous} and satisfies the \textit{log-H\"{o}lder decay condition};
\item We define the class of variable exponents $\mathcal{P}_{0}^{\log}(\mathbb{R}^{n})$ as
\begin{equation*}
 \mathcal{P}_{0}^{\log}(\mathbb{R}^{n}):=\left\{p(\cdot)\in\mathcal{P}(\mathbb{R}^{n}):\ \  p(\cdot)\ \text{is globally log-H\"{o}lder continuous in}\  \mathbb{R}^n\right\}.
\end{equation*}
\end{itemize}
\end{definition}

For any $p(\cdot)\in \mathcal{P}_{0}^{\log}(\mathbb{R}^{n})$, we have the following results in terms of the Hardy--Littlewood maximal function and  Riesz transforms (cf. \cite{CF13}, Theorem 3.16 and Theorem 5.42;\cite{DHHR11}, Theorem 4.3.8 and Corollary 6.3.10).

\begin{lemma}\label{le2.4}
Let $p(\cdot)\in \mathcal{P}^{\log}_{0}(\mathbb{R}^{n})$ with $1<p^{-}\leq p^{+}<+\infty$. Then  for any $f\in L^{p(\cdot)}(\mathbb{R}^{n})$, there exists a positive constant $C$ such that
\begin{equation}\label{eq2.3}
   \|\mathcal{M} (f)\|_{L^{p(\cdot)}}\leq C\|f\|_{L^{p(\cdot)}},
\end{equation}
where $\mathcal{M}$ is the Hardy--Littlewood maximal function defined by
\begin{equation*}
  \mathcal{M}(f)(x):=\sup_{x\in B}\frac{1}{|B|}\int_{B}|f(y)|dy,
\end{equation*}
and $B\subset \mathbb{R}^{n}$ is an open ball with center $x$. Furthermore,
\begin{equation}\label{eq2.4}
   \|\mathcal{R}_{j}(f)\|_{L^{p(\cdot)}}\leq C\|f\|_{L^{p(\cdot)}} \ \ \text{for any}\ \ 1\leq j\leq n,
\end{equation}
where $\mathcal{R}_{j}$ ($1\leq j\leq n$) are the usual Riesz transforms, i.e. $\mathcal{F}\left(\mathcal{R}_{j}f\right)(\xi)=-\frac{i\xi_{j}}{|\xi|}\mathcal{F}(f)(\xi)$.
\end{lemma}

\subsection{Littlewood--Paley decomposition and Fourier--Besov spaces}
Let us briefly recall some basic facts on Littlewood--Paley dyadic decomposition theory.  More details may be found in Chap. 2 and Chap. 3 in the book \cite{BCD11}. Choose a smooth radial non-increasing function $\chi$  with $\operatorname{Supp}\chi\subset B(0,\frac{4}{3})$ and $\chi\equiv1$ on $B(0,\frac{3}{4})$. Set $\varphi(\xi)=\chi(\frac{\xi}{2})-\chi(\xi)$. It is not difficult to check that
 $\varphi$ is supported in the shell $\{\xi\in\mathbb{R}^{n},\ \frac{3}{4}\leq
|\xi|\leq \frac{8}{3}\}$, and
\begin{align*}
   \sum_{j\in\mathbb{Z}}\varphi(2^{-j}\xi)=1, \ \ \  \xi\in\mathbb{R}^{n}\backslash\{0\}.
\end{align*}
 For any $f\in\mathcal{S}'(\mathbb{R}^{n})$, the homogeneous dyadic blocks $\Delta_{j}$ ($j\in\mathbb {Z}$) and the low-frequency cutoff operator  $ S_{j}$ are defined for all $j\in\mathbb{Z}$ by
\begin{align*}
 \Delta_{j}f: =\mathcal{F}^{-1}(\varphi(2^{-j}\cdot)\mathcal{F}f), \ \ \ S_{j}f(x): =\mathcal{F}^{-1}(\chi(2^{-j}\cdot)\mathcal{F}f).
\end{align*}
Let $\mathcal{S}'_{h}(\mathbb{R}^{n})$ be the space of tempered distribution $f\in\mathcal{S}'(\mathbb{R}^{n})$ such that
$$
\lim_{j\rightarrow -\infty} S_{j}f(x)=0.
$$
Then one has the unit decomposition for any tempered distribution $f\in\mathcal{S}'_{h}(\mathbb{R}^{n})$:
\begin{align}\label{eq2.5}
 f=\sum_{j\in\mathbb{Z}}\Delta_{j}f.
\end{align}
The above homogeneous dyadic block $\Delta_{j}$ and
the partial summation operator $S_{j}$  satisfy the following quasi-orthogonal properties: for any $f, g\in\mathcal{S}'(\mathbb{R}^{n})$, one has
\begin{align}\label{eq2.6}
  \Delta_{i}\Delta_{j}f\equiv0\ \ \ \text{if}\ \ \ |i-j|\geq 2\ \ \ \text{and}\ \ \
  \Delta_{i}(S_{j-1}f\Delta_{j}g)\equiv0 \ \ \ \text{if}\ \ \ |i-j|\geq 5.
\end{align}
\smallbreak

Applying the above decomposition, the homogeneous Fourier--Besov space $F\dot{B}_{p,r}^s(\mathbb{R}^{n})$ and the Chemin--Lerner type space $ \tilde{L}^{\lambda}(0,T; F\dot{B}^{s}_{p,r}(\mathbb{R}^{n}))$ can be defined as follows:

\begin{definition}\label{def2.5}
Let $s\in \mathbb{R}$ and $1\leq p,r\leq\infty$, the homogeneous Fourier--Besov space $F\dot{B}^{s}_{p,r}(\mathbb{R}^{n})$ is defined to be the set of all tempered distributions $f\in \mathcal{S}'_{h}(\mathbb{R}^{n})$ such that $\hat{f} \in L_{loc}^1(\mathbb{R}^n)$ and the following norm is finite:
\begin{equation*}
  \|f\|_{F\dot{B}^{s}_{p,r}}:= \begin{cases} \left(\sum_{j\in\mathbb{Z}}2^{jsr}\|\widehat{\Delta_{j}f}\|_{L^{p}}^{r}\right)^{\frac{1}{r}}
  \ \ &\text{if}\ \ 1\leq r<\infty,\\
  \sup_{j\in\mathbb{Z}}2^{js}\|\widehat{\Delta_{j}f}\|_{L^{p}}\ \
  &\text{if}\ \
  r=\infty.
 \end{cases}
\end{equation*}
\end{definition}

\begin{definition}\label{def2.6} For $0<T\leq\infty$, $s\in \mathbb{R}$ and $1\leq p, r, \lambda\leq\infty$, we set  {\rm{(}}with the usual convention if
$r=\infty${\rm{)}}:
$$
  \|f\|_{\mathcal{L}^{\lambda}_{T}(F\dot{B}^{s}_{p,r})}:=\big(\sum_{j\in\mathbb{Z}}2^{jsr}\|\widehat{\Delta_{j}f}\|_{L^{\lambda}_{T}(L^{p})}^{r}\big)^{\frac{1}{r}}.
$$
Then we define the space $\mathcal{L}^{\lambda}(0,T; F\dot{B}^{s}_{p,r}(\mathbb{R}^{n}))$  as the set of temperate distributions $f$ over $(0,T)\times \mathbb{R}^{n}$ such that $\lim_{j\rightarrow -\infty}S_{j}f=0$ in $\mathcal{S}'((0,T)\times\mathbb{R}^{n})$ and $\|f\|_{\mathcal{L}^{\lambda}_{T}(F\dot{B}^{s}_{p,r})}<\infty$.
\end{definition}

Based on the properties of Fourier transform, we have the following Bernstein type lemma (cf. Lemma 2.3 in \cite{ZZ19}).
\begin{lemma}\label{le2.7}
Let $\mathcal{B}$ be a ball and $\mathcal{C}$ a ring of $\mathbb{R}^n$. A constant C exists so that for any positive real number $\lambda$, any nonnegative integer $k$, any homogeneous function $\sigma$ of degree $m$ smooth outside of zero and any couple of real numbers $(p,q)$ with $1\leq q \leq p$,  there hold
\begin{itemize}

  \item[(\romannumeral1)] $\operatorname{supp} \widehat{u} \subset \lambda \mathcal{B} \Rightarrow \sup_{|\alpha|=k}\|\widehat{\partial^{\alpha}u}\|_{L^q} \leq C^{k+1} \lambda^{k+n(\frac{1}{q}-\frac{1}{p})} \|\widehat{u}\|_{L^p}$;

  \item[(\romannumeral2)] $\operatorname{supp}  \widehat{u} \subset \lambda \mathcal{C} \Rightarrow C^{-k-1} \lambda^{k} \|\widehat{u}\|_{L^p} \leq \sup_{|\alpha|=k}\|\widehat{\partial^{\alpha}u}\|_{L^p} \leq C^{k+1} \lambda^{k} \|\widehat{u}\|_{L^p}$;

  \item[(\romannumeral3)] $\operatorname{supp} \widehat{u} \subset \lambda \mathcal{C} \Rightarrow \|\widehat{\sigma(D) u}\|_{L^q} \leq C(\sigma,m) \lambda^{m+n(\frac{1}{q}-\frac{1}{p})} \|\widehat{u}\|_{L^p}.$
\end{itemize}
\end{lemma}

Let us state some properties of the homogeneous Fourier--Besov spaces (cf. Lemma 2.5 in \cite{ZZ19}).
\begin{lemma}\label{le2.8}
The following properties hold in the homogeneous Fourier--Besov spaces:
\begin{itemize}
  \item[(\romannumeral1)] There exists a universal constant $C$ such that
    \begin{align*}
     C^{-1}\|u\|_{F\dot{B}^{s}_{p,r}} \leq  \|\nabla u\|_{F\dot{B}^{s-1}_{p,r}} \leq C\|u\|_{F\dot{B}^{s}_{p,r}}.
    \end{align*}
  \item[(\romannumeral2)] For $1\leq p_1 \leq p_2\leq\infty$ and $1\leq r_1 \leq r_2\leq\infty$, we have the continuous imbedding $F\dot{B}^{s}_{p_2,r_2}(\mathbb{R}^n) \hookrightarrow F\dot{B}^{s-n(\frac{1}{p_1}-\frac{1}{p_2})}_{p_1,r_1}(\mathbb{R}^n)$.
  \item[(\romannumeral3)] For $s_1,s_2 \in \mathbb{R}$ such that $s_1<s_2$ and $\theta\in (0,1)$, there exists a constant $C$ such that
    \begin{align*}
     \|u\|_{F\dot{B}^{s_1\theta+s_2(1-\theta)}_{p,r}} \leq  C\|u\|^{\theta}_{F\dot{B}^{s_1}_{p,r}} \|u\|^{1-\theta}_{F\dot{B}^{s_2}_{p,r}}.
    \end{align*}
\end{itemize}
\end{lemma}

We also need to establish the following generalized product estimates between two
distributions in the homogeneous Fourier--Besov spaces.

\begin{lemma}\label{le2.9}
Let $s>0$ and $1 \leq p, p_1, p_2 \leq \infty$ such that
\begin{equation*}
  1+\frac{1}{p}=\frac{1}{p_{1}}+\frac{1}{p_{2}}.
\end{equation*}
Then there exists a constant $C>0$ such that
\begin{equation}\label{eq2.7}
\|uv\|_{F\dot{B}_{p,1}^{s}} \leq C\left( \|u\|_{F\dot{B}_{p_1,1}^{s}} \|v\|_{F\dot{B}_{p_2,1}^{0}}+\|u\|_{F\dot{B}_{p_2,1}^{0}} \|v\|_{F\dot{B}_{p_1,1}^{s}}\right).
\end{equation}
\end{lemma}
\begin{proof}
Thanks to Bony's paraproduct decomposition (see \cite{B81}), we have
\begin{equation*}
  uv=T_u v+T_v u+R(u,v),
\end{equation*}
where
\begin{equation*}
\begin{cases}
  T_{u}v:=\sum_{j\in\mathbb{Z}}\sum_{k\leq j-2}\Delta_{k}u\Delta_{j}v=\sum_{j\in\mathbb{Z}}S_{j-1}u\Delta_{j}v,\\
  R(u,v):=\sum_{j\in\mathbb{Z}}\sum_{|j-j'| \leq 2}\Delta_{j}u\Delta_{j'}v.
\end{cases}
\end{equation*}
For $T_{u}v$,  using the condition $s>0$, one can derive from \eqref{eq2.6} that
\begin{align}\label{eq2.8}
    \|\mathcal{F}[\Delta_j{T_{u}v}]\|_{L^{p}}&\lesssim \|\sum_{|k-j|\leq 4} \mathcal{F}[\Delta_j(S_{k-1}u \Delta_k v)]\|_{L^{p}} \nonumber \\
    &\lesssim \sum_{|k-j|\leq 4}\sum_{k' \leq k-2}\|\mathcal{F}[\Delta_{k'}u]\|_{L^{p_{1}}} \|\mathcal{F}[\Delta_k v]\|_{L^{p_2}} \nonumber \\
     &\lesssim \sum_{|k-j|\leq 4}2^{-ks}\sum_{k' \leq k-2}2^{k's}\|\mathcal{F}[\Delta_{k'}u]\|_{L^{p_{1}}} \|\mathcal{F}[\Delta_k v]\|_{L^{p_2}} \nonumber \\
    &\lesssim 2^{-js}\sum_{|k-j|\leq 4}  \|\mathcal{F}[\Delta_k' v]\|_{L^{p_2}}\|u\|_{F\dot{B}_{p_1,1}^{s}}\nonumber \\
     &\lesssim 2^{-js}d_{j}\|u\|_{F\dot{B}_{p_1,1}^{s}} \|v\|_{F\dot{B}_{p_2,1}^{0}}.
\end{align}
Similarly, for $T_{v} u$, we have
\begin{align}\label{eq2.9}
    \|\mathcal{F}[\Delta_j{T_{v}u}]\|_{L^{p}}\lesssim 2^{-js}d_{j}\|u\|_{F\dot{B}_{p_2,1}^{0}} \|v\|_{F\dot{B}_{p_1,1}^{s}}.
\end{align}
For the remaining term $R(u,v)$,  we can directly bound it as
\begin{align}\label{eq2.10}
    \|\mathcal{F}[\Delta_j{R(u,v)}]\|_{L^{p}}&\lesssim \|\sum_{k \geq j-4}\sum_{|k-k'| \leq 2} \mathcal{F}[\Delta_j(\Delta_k u \Delta_{k'} v)]\|_{L^{p}}\nonumber \\
    &\lesssim \sum_{k \geq j-4}\sum_{|k-k'| \leq 2} \|\mathcal{F}[\Delta_k u]\|_{L^{p_1}} \|\mathcal{F}[\Delta_{k'} v]\|_{L^{p_2}} \nonumber \\
    &\lesssim 2^{-js}\sum_{k \geq j-4}2^{-(k-j)s}2^{ks}\|\mathcal{F}[\Delta_k u]\|_{L^{p_1}}\|v\|_{F\dot{B}_{p_2,1}^{0}}\nonumber \\
    &\lesssim 2^{-js}d_{j}\|u\|_{F\dot{B}_{p_1,1}^{s}} \|v\|_{F\dot{B}_{p_2,1}^{0}}.
\end{align}
Combining the above estimates \eqref{eq2.8}--\eqref{eq2.10} together, we get \eqref{eq2.7}. We have completed the proof of Lemma \ref{le2.9}.
\end{proof}

\begin{lemma}\label{le2.10}
Let $s_1, s_2 \in \mathbb{R}$ and $1 \leq p_1,p_2 \leq \infty$.  Assume that
\begin{equation*}
   s_1 \leq n\min\{1-\frac{1}{p_1}, 1-\frac{1}{p_2}\}, \ \ s_2\leq n(1-\frac{1}{p_2})
\end{equation*}
and
\begin{equation*}
  s_1+s_2>\max\{0, n(1-\frac{1}{p_1}-\frac{1}{p_2})\}.
\end{equation*}
Then there exists a constant $C>0$ such that
\begin{equation}\label{eq2.11}
\|uv\|_{F\dot{B}_{p_1,1}^{s_1+s_2-n(1-\frac{1}{p_2})}} \leq C \|u\|_{F\dot{B}_{p_1,1}^{s_1}} \|v\|_{F\dot{B}_{p_2,1}^{s_2}}.
\end{equation}
\end{lemma}
\begin{proof}
For $T_v u$, applying Lemma \ref{le2.7} gives us to
\begin{align}\label{eq2.12}
    \|\mathcal{F}[\Delta_j&{T_v u}]\|_{L^{p_1}}\lesssim \|\sum_{|k-j|\leq 4} \mathcal{F}[\Delta_j(S_{k-1}v \Delta_k u)]\|_{L^{p_1}} \nonumber \\
    &\lesssim \sum_{|k-j|\leq 4}\sum_{k' < k-1}\|\mathcal{F}[\Delta_{k'}v]\|_{L^{1}} \|\mathcal{F}[\Delta_k u]\|_{L^{p_1}} \nonumber \\
    &\lesssim \sum_{|k-j|\leq 4}  \sum_{k' \leq k-2} 2^{k'n (1-\frac{1}{p_2})}\|\mathcal{F}[\Delta_{k'}v]\|_{L^{p_2}} \|\mathcal{F}[\Delta_k u]\|_{L^{p_1}} \nonumber \\
    &\lesssim \sum_{|k-j|\leq 4}  \sum_{k' \leq k-2} 2^{k'(n(1-\frac{1}{p_2})-s_2)}  2^{-ks_1}2^{k' s_2}\|\mathcal{F}[\Delta_k' v]\|_{L^{p_2}}2^{k s_1}\|\mathcal{F}[\Delta_{k}u]\|_{L^{p_1}}  \nonumber \\
    &\lesssim d_j 2^{-j (s_1+s_2-n(1-\frac{1}{p_2}))} \|u\|_{F\dot{B}_{p_1,1}^{s_1}} \|v\|_{F\dot{B}_{p_2,1}^{s_2}}.
\end{align}
For $T_u v$, we split it into the following two cases: in the case of $p_1 \leq p_2$,  one gets
\begin{align}\label{eq2.13}
     \|\mathcal{F}[\Delta_j&{T_u v}]\|_{L^{p_1}}\lesssim 2^{jn(\frac{1}{p_1}-\frac{1}{p_2})}\|\sum_{|k-j|\leq 4} \mathcal{F}[\Delta_j(S_{k-1}u \Delta_k v)]\|_{L^{p_2}} \nonumber \\
     &\lesssim 2^{jn(\frac{1}{p_1}-\frac{1}{p_2})}\sum_{|k-j|\leq 4} \sum_{k' \leq k-2}  \|\mathcal{F}[\Delta_{k'}u]\|_{L^{1}} \|\mathcal{F}[\Delta_k v]\|_{L^{p_2}} \nonumber \\
     &\lesssim 2^{jn(\frac{1}{p_1}-\frac{1}{p_2})}\sum_{|k-j|\leq 4}  \sum_{k'\leq k-2} 2^{k'n (1-\frac{1}{p_1})}\|\mathcal{F}[\Delta_{k'}u]\|_{L^{p_1}} \|\mathcal{F}[\Delta_k v]\|_{L^{p_2}} \nonumber \\
     &\lesssim 2^{jn(\frac{1}{p_1}-\frac{1}{p_2})}\sum_{|k-j|\leq 4}  \sum_{k' \leq k-2}  2^{k'(n(1-\frac{1}{p_1})-s_1)}  2^{-ks_2}2^{k' s_1}\|\mathcal{F}[\Delta_{k'}u]\|_{L^{p_1}} 2^{k s_2}\|\mathcal{F}[\Delta_k v]\|_{L^{p_2}} \nonumber \\
     &\lesssim d_j 2^{-j (s_1+s_2-n(1-\frac{1}{p_2}))} \|u\|_{F\dot{B}_{p_1,1}^{s_1}} \|v\|_{F\dot{B}_{p_2,1}^{s_2}};
\end{align}
while in the case of $p_1 > p_2$, using the condition $ s_1\leq n\min\{1-\frac{1}{p_1}, 1-\frac{1}{p_2}\}$, one can deduce that
\begin{align}\label{eq2.14}
    \|&\mathcal{F}[\Delta_j{T_u v}]\|_{L^{p_1}}\lesssim \|\sum_{|k-j|\leq 4} \mathcal{F}[\Delta_j(S_{k-1}u \Delta_k v)]\|_{L^{p_1}} \nonumber \\
     &\lesssim \sum_{|k-j|\leq 4} \sum_{k' \leq k-2}  \|\mathcal{F}[\Delta_{k'}u]\|_{L^{\tilde{p}}} \|\mathcal{F}[\Delta_k v]\|_{L^{p_2}}  \ \  \Big( \frac{1}{p_1}\leq\frac{1}{p_1}+1-\frac{1}{p_2}=\frac{1}{\tilde{p}} \Big)\nonumber \\
     &\lesssim \sum_{|k-j|\leq 4}  \sum_{k' \leq k-2} 2^{k'n {(\frac{1}{\tilde{p}}-\frac{1}{p_1})}}\|\mathcal{F}[\Delta_{k'}u]\|_{L^{p_1}} \|\mathcal{F}[\Delta_k v]\|_{L^{p_2}} \nonumber \\
     &\lesssim \sum_{|k-j|\leq 4}  \sum_{k' \leq k-2}  2^{k'(n(1-\frac{1}{p_2})-s_1)}  2^{-ks_2}2^{k' s_1}\|\mathcal{F}[\Delta_{k'}u]\|_{L^{p_1}} 2^{k s_2}\|\mathcal{F}[\Delta_k v]\|_{L^{p_2}} \nonumber \\
     &\lesssim d_j 2^{-j (s_1+s_2-n(1-\frac{1}{p_2}))}  \|u\|_{F\dot{B}_{p_1,1}^{s_1}} \|v\|_{F\dot{B}_{p_2,1}^{s_2}}.
\end{align}
For the remaining term $R(u,v)$, we also need to consider the following two cases: in the case of  $\frac{1}{p_1}+\frac{1}{p_2}\geq 1$, we choose $\bar{p}$ such that $1+\frac{1}{\bar{p}}=\frac{1}{p_1}+\frac{1}{p_2}$, then using the condition $s_1+s_2>0$ to obtain that
\begin{align}\label{eq2.15}
    \|\mathcal{F}[\Delta_j&{R(u,v)}]\|_{L^{p_1}}\lesssim 2^{jn(1-\frac{1}{p_2})} \|\sum_{k \geq j-4}\sum_{|k-k'| \leq 2} \mathcal{F}[\Delta_j(\Delta_k u \Delta_{k'} v)]\|_{L^{\bar{p}}}  \ \  \Big( p_1 \leq\bar{p} \Big)\nonumber \\
    &\lesssim 2^{jn(1-\frac{1}{p_2})} \sum_{k \geq j-4}\sum_{|k-k'| \leq 2} \|\mathcal{F}[\Delta_k u]\|_{L^{p_1}} \|\mathcal{F}[\Delta_{k'} v]\|_{L^{p_2}} \nonumber \\
    &\lesssim 2^{jn(1-\frac{1}{p_2})} \sum_{k \geq j-4}\sum_{|k-k'| \leq 2}  2^{-k (s_1+s_2)}2^{k s_1} \|\mathcal{F}[\Delta_k u]\|_{L^{p_1}} 2^{k' s_2}\|\mathcal{F}[\Delta_{k'} v]\|_{L^{p_2}} \nonumber \\
    &\lesssim d_j 2^{-j (s_1+s_2-n(1-\frac{1}{p_2}))} \|u\|_{F\dot{B}_{p_1,1}^{s_1}} \|v\|_{F\dot{B}_{p_2,1}^{s_2}};
\end{align}
while in the case of $\frac{1}{p_1}+\frac{1}{p_2}<1$, using the condition $s_1+s_2>n(1-\frac{1}{p_1}-\frac{1}{p_2})$, we obtain
\begin{align}\label{eq2.16}
    \|\mathcal{F}[\Delta_j&{R(u,v)}]\|_{L^{p_1}}\lesssim 2^{j\frac{n}{p_1}} \|\sum_{k \geq j-4}\sum_{|k-k'| \leq 2} \mathcal{F}[\Delta_j(\Delta_k u \Delta_{k'} v)]\|_{L^\infty} \nonumber \\
    &\lesssim 2^{j\frac{n}{p_1}} \sum_{k \geq j-4}\sum_{|k-k'| \leq 2} \|\mathcal{F}[\Delta_k u]\|_{L^{p'_2}} \|\mathcal{F}[\Delta_{k'} v]\|_{L^{p_2}} \ \  \Big( \frac{1}{p_2}+\frac{1}{p'_2}=1 \Big) \nonumber \\
    &\lesssim 2^{j\frac{n}{p_1}} \sum_{k \geq j-4}\sum_{|k-k'| \leq 2} 2^{kn(1-\frac{1}{p_1}-\frac{1}{p_2})}\|\mathcal{F}[\Delta_k u]\|_{L^{p_1}} \|\mathcal{F}[\Delta_{k'} v]\|_{L^{p_2}}  \ \ (p_1>p'_2) \nonumber \\
    &\lesssim 2^{j\frac{n}{p_1}} \sum_{k \geq j-4}\sum_{|k-k'| \leq 2} 2^{k(n(1-\frac{1}{p_1}-\frac{1}{p_2})-s_{1}-s_{2})}
      2^{k s_1}\|\mathcal{F}[\Delta_k u]\|_{L^{p_1}} 2^{k' s_2}\|\mathcal{F}[\Delta_{k'} v]\|_{L^{p_2}} \nonumber \\
    &\lesssim d_j 2^{-j (s_1+s_2-n(1-\frac{1}{p_2}))}  \|u\|_{F\dot{B}_{p_1,1}^{s_1}} \|v\|_{F\dot{B}_{p_2,1}^{s_2}}.
\end{align}
Putting the above estimates \eqref{eq2.12}--\eqref{eq2.16} together, we get \eqref{eq2.11}. We have completed the proof of Lemma \ref{le2.10}.

\end{proof}

\subsection{Variable Fourier--Besov spaces}
 The Besov spaces $\dot{B}^{s(\cdot)}_{p(\cdot),r}(\mathbb{R}^{n})$ with variable regularity and integrability were introduced  by Vyb\'{r}al \cite{V09},  Kempka \cite{K09} and Almeida--H\"{a}st\"{o} \cite{AH10}. Inspired by the ideas in \cite{AH10},  Ru \cite{R18} introduced the variable Fourier--Besov spaces to establish the well-posedness of the 3D Navier--Stokes equations in such spaces.  Let us recall its  definition. Let $p(\cdot),r(\cdot)\in \mathcal{P}_{0}(\mathbb{R}^{n})$. The mixed Lebesgue-sequence space $\ell^{r(\cdot)}(L^{p(\cdot)})$ is defined on sequences of $L^{p(\cdot)}$-functions by the modular
\begin{equation}\label{eq2.17}
 \varrho_{\ell^{r(\cdot)}(L^{p(\cdot)})}\left(\{f_{j}\}_{j\in\mathbb{Z}}\right):=\sum_{j\in \mathbb{Z}}\inf\left\{\lambda_{j}>0\Big{|} \varrho_{p(\cdot)}\left(\frac{f_{j}}{\lambda_{j}^{\frac{1}{r(\cdot)}}}\right) \leq 1 \right\}.
\end{equation}
The norm is defined from this as usual:
\begin{equation*}
 \left\|\{f_{j}\}_{j\in\mathbb{Z}}\right\|_{\ell^{r(\cdot)}(L^{p(\cdot)})}:=\inf\left\{\mu>0\Big{|} \varrho_{\ell^{r(\cdot)}(L^{p(\cdot)})}\left(\{\frac{f_{j}}{\mu}\}_{j\in\mathbb{Z}}\right) \leq 1 \right\}.
\end{equation*}
Since we assume that $r^{+}<\infty$,  it holds that
\begin{equation*}
 \varrho_{\ell^{r(\cdot)}(L^{p(\cdot)})}\left(\{f_{j}\}_{j\in\mathbb{Z}}\right)=\sum_{j\in\mathbb{Z}}\left\||f_{j}|^{r(\cdot)}\right\|_{\frac{p(\cdot)}{r(\cdot)}}.
\end{equation*}

Let us recall the definition of the  Fourier--Besov space $F\dot{B}^{s(\cdot)}_{p(\cdot),r(\cdot)}(\mathbb{R}^{n})$ with variable exponents.

\begin{definition}\label{de2.11}
Let $s(\cdot), p(\cdot),r(\cdot)\in \mathcal{P}_{0}^{\log}(\mathbb{R}^{n})$. We define the variable exponent Fourier--Besov space $F\dot{B}^{s(\cdot)}_{p(\cdot),r(\cdot)}(\mathbb{R}^{n})$  as the collection of all distributions $f\in \mathcal{S}'_{h}(\mathbb{R}^{n})$ such that
\begin{equation}\label{eq2.18}
\left\|f\right\|_{F\dot{B}^{s(\cdot)}_{p(\cdot),r(\cdot)}}:=\left\|\{2^{js(\cdot)}\varphi_{j}\widehat{f}\}_{j\in\mathbb{Z}}\right\|_{\ell^{r(\cdot)}(L^{p(\cdot)})}<\infty.
\end{equation}
\end{definition}

We also need the following Chemin--Lerner type mixed time-space spaces in terms of the Fourier--Besov space with variable exponent.

\begin{definition}\label{de2.12}
Let $s(\cdot), p(\cdot)\in \mathcal{P}_{0}^{\log}(\mathbb{R}^{n})$ and  $1\leq \rho, r\leq\infty$. We define the Chemin--Lerner mixed time-space $\mathcal{L}^{\rho}(0,T; \dot{B}^{s(\cdot)}_{p(\cdot),r}(\mathbb{R}^{n}))$ as the completion of $\mathcal{C}([0,T]; \mathcal{S}(\mathbb{R}^{n}))$ by the norm
\begin{equation}\label{eq2.19}
\left\|f\right\|_{\mathcal{L}^{\rho}_{T}(F\dot{B}^{s(\cdot)}_{p(\cdot),r})}
:=\left(\sum_{j\in\mathbb{Z}}\left\|2^{js(\cdot)}\varphi_{j}\widehat{f}\right\|^{r}_{L^{\rho}_{T}(L^{p(\cdot)})}\right)^{\frac{1}{r}}<\infty
\end{equation}
with the usual change if $\rho=\infty$
or $r=\infty$. If $T=\infty$, we use $\left\|f\right\|_{\mathcal{L}^{\rho}_{t}(F\dot{B}^{s(\cdot)}_{p(\cdot),r})}$ instead of $\left\|f\right\|_{\mathcal{L}^{\rho}_{\infty}(F\dot{B}^{s(\cdot)}_{p(\cdot),r})}$.
\end{definition}

Finally, let us recall the following existence and uniqueness result for an abstract operator equation in a generic Banach space (cf. \cite{L02}, Theorem 13.2).
\begin{proposition}\label{pro2.13}
Let $\mathcal{X}$ be a Banach space and
$\mathcal{B}:\mathcal{X}\times\mathcal{X}\rightarrow\mathcal{X}$  a
bilinear bounded operator. Assume that for any $u,v\in
\mathcal{X}$, we have
$$
  \|\mathcal{B}(u,v)\|_{\mathcal{X}}\leq
  C\|u\|_{\mathcal{X}}\|v\|_{\mathcal{X}}.
$$
Then for any $y\in \mathcal{X}$ such that $\|y\|_{\mathcal{X}}\leq
\eta<\frac{1}{4C}$, the equation $u=y+\mathcal{B}(u,u)$ has a
solution $u$ in $\mathcal{X}$. Moreover, this solution is the only
one such that $\|u\|_{\mathcal{X}}\leq 2\eta$, and depends
continuously on $y$ in the following sense: if
$\|\widetilde{y}\|_{\mathcal{X}}\leq \eta$,
$\widetilde{u}=\widetilde{y}+\mathcal{B}(\widetilde{u},\widetilde{u})$
and $\|\widetilde{u}\|_{\mathcal{X}}\leq 2\eta$, then
$$
  \|u-\widetilde{u}\|_{\mathcal{X}}\leq \frac{1}{1-4\eta
  C}\|y-\widetilde{y}\|_{\mathcal{X}}.
$$
\end{proposition}

\section{Solvability of the heat equation in variable Fourier--Besov spaces}

The basic heat equation reads
\begin{equation}\label{eq3.1}
\begin{cases}
\partial_{t}u-\Delta u=f(t,x), \ \ \ t\in \mathbb{R}^{+}, x\in \mathbb{R}^{n},\\
u(0,x)=u_{0}(x), \ \ \  x\in \mathbb{R}^{n}.
\end{cases}
\end{equation}
We get the following linear estimates of solutions to the heat equation in the framework of Fourier--Besov space $F\dot{B}^{s(\cdot)}_{p(\cdot),r}(\mathbb{R}^{n})$ with variable regularity and integrability indices.
\begin{theorem}\label{th3.1}
Let $0<T\leq \infty$, $1\leq \rho,r\leq\infty$ and $s(\cdot), p(\cdot)\in \mathcal{P}_{0}^{\log}(\mathbb{R}^{n})$. Assume that $u_{0}\in F\dot{B}^{s(\cdot)}_{p(\cdot),r}(\mathbb{R}^{n})$, $f\in\mathcal{L}^{\rho}(0,T; F\dot{B}^{s(\cdot)+\frac{2}{\rho}-2}_{p(\cdot),r}(\mathbb{R}^{n}))$. Then the heat equation \eqref{eq3.1} has a unique solution
$$
  u\in \mathcal{L}^{\infty}(0,T; F\dot{B}^{s(\cdot)}_{p(\cdot),r}(\mathbb{R}^{n}))\cap \mathcal{L}^{\rho}(0,T; F\dot{B}^{s(\cdot)+\frac{2}{\rho}}_{p(\cdot),r}(\mathbb{R}^{n})).
$$
Moreover, for any $\rho_{1}\in[\rho,\infty]$, there exists a positive constant $C$ depending only on $n$ such that
\begin{equation}\label{eq3.2}
\|u\|_{\mathcal{L}^{\rho_{1}}_{T}(F\dot{B}^{s(\cdot)+\frac{2}{\rho_{1}}}_{p(\cdot),r})}\leq C\left(\|u_{0}\|_{ F\dot{B}^{s(\cdot)}_{p(\cdot),r}}+\|f\|_{\mathcal{L}^{\rho}_{T}(F\dot{B}^{s(\cdot)+\frac{2}{\rho}-2}_{p(\cdot),r})}\right).
\end{equation}
\end{theorem}
\begin{proof}
Since $u_{0}$ and $f$ are temperate distributions, the heat equation \eqref{eq3.1} admits a unique solution $u$ in $\mathcal{S}'((0,T)\times \mathbb{R}^{n})$, which satisfies
\begin{equation}\label{eq3.3}
\widehat{u}(t,\xi)=e^{-t|\xi|^{2}}\widehat{u}_{0}(\xi)+\int_{0}^{t}e^{-(t-\tau)|\xi|^{2}}\widehat{f}(\tau,\xi)d\tau.
\end{equation}
Notice that $\phi_{j}$ is supported in the ring $\{\frac{3}{4}2^{j}\leq |\xi|\leq \frac{8}{3}2^{j}\}$, thus there exists a constant $\kappa>0$ such that
\begin{align}\label{eq3.4}
  \|2^{js(\cdot)}\phi_{j}\widehat{u}(t,\xi)\|_{L^{\rho_{1}}_{T}(L^{p(\cdot)}_{\xi})}&\leq \|e^{-\kappa t2^{2j}}2^{js(\cdot)}\phi_{j}\widehat{u}_{0}(\xi)\|_{L^{\rho_{1}}_{T}(L^{p(\cdot)}_{\xi})}\nonumber\\
  &+\left\|\int_{0}^{t}e^{-\kappa(t-\tau)2^{2j}}2^{js(\cdot)}\phi_{j}\widehat{f}(\tau,\xi)d\tau\right\|_{L^{\rho_{1}}_{T}(L^{p(\cdot)}_{\xi})}.
\end{align}
By Young's inequality with respect to the time variable, we get
\begin{align}\label{eq3.5}
  \|2^{js(\cdot)}\phi_{j}\widehat{u}(t,\xi)\|_{L^{\rho_{1}}_{T}(L^{p(\cdot)}_{\xi})}&\leq \left(\frac{1-e^{-\kappa 2^{2j}\rho_{1}T}}{\kappa2^{2j}\rho_{1}}\right)^{\frac{1}{\rho_{1}}}\|2^{js(\cdot)}\phi_{j}\widehat{u}_{0}(\xi)\|_{L^{p(\cdot)}_{\xi}}\nonumber\\
  &+ \left(\frac{1-e^{-\kappa 2^{2j}\rho_{2}T}}{\kappa2^{2j}\rho_{2}}\right)^{\frac{1}{\rho_{2}}}\left\|2^{js(\cdot)}\phi_{j}\widehat{f}(\tau,\xi)\right\|_{L^{\rho}_{T}(L^{p(\cdot)}_{\xi})},
\end{align}
where $1+\frac{1}{\rho_{1}}=\frac{1}{\rho}+\frac{1}{\rho_{2}}$. Finally, taking $\ell^{r}$-norm into consideration, we conclude that
\begin{align}\label{eq3.6}
  \|u\|_{\mathcal{L}^{\rho_{1}}_{T}(F\dot{B}^{s(\cdot)+\frac{2}{\rho_{1}}}_{p(\cdot),r})}&\leq \left[\sum_{j\in\mathbb{Z}}\left(\frac{1-e^{-\kappa 2^{2j}\rho_{1}T}}{\kappa\rho_{1}}\right)^{\frac{r}{\rho_{1}}}\left(\|2^{js(\cdot)}\phi_{j}\widehat{u}_{0}(\xi)\|_{L^{p(\cdot)}_{\xi}}\right)^{r}\right]^{\frac{1}{r}}\nonumber\\
  &+ \left[\sum_{j\in\mathbb{Z}}\left(\frac{1-e^{-\kappa2^{2j} \rho_{2}T}}{\kappa \rho_{2}}\right)^{\frac{r}{\rho_{2}}}\left(\left\|2^{j(s(\cdot)-2+\frac{2}{\rho})}\phi_{j}\widehat{f}(\tau,\xi)\right\|_{L^{\rho}_{T}(L^{p(\cdot)}_{\xi})}\right)^{r}\right]^{\frac{1}{r}}\nonumber\\
  &\leq C\left(\|u_{0}\|_{ F\dot{B}^{s(\cdot)}_{p(\cdot),r}}+\|f\|_{\mathcal{L}^{\rho}_{T}(F\dot{B}^{s(\cdot)+\frac{2}{\rho}-2}_{p(\cdot),r})}\right).
\end{align}
Thus we get $u\in \mathcal{L}^{\infty}(0,T; F\dot{B}^{s(\cdot)}_{p(\cdot),r}(\mathbb{R}^{n}))\cap \mathcal{L}^{\rho}(0,T;F\dot{B}^{s(\cdot)+\frac{2}{\rho}}_{p(\cdot),r}(\mathbb{R}^{n}))$ and the inequality \eqref{eq3.2}.
\end{proof}

\section{Proof of Theorem \ref{th1.1}}

In order to apply Proposition \ref{pro2.13} to the Navier--Stokes equations \eqref{eq1.1}, we need to get rid of the pressure $\pi$, and for this purpose we apply to \eqref{eq1.1} the Leray projector $\mathbb{P}$ defined by $\mathbb{P}u=u+\nabla(-\Delta)^{-1}\nabla\cdot u$, i.e., $\mathbb{P}$ is the $3\times 3$
matrix pseudo-differential operator with the
symbol $(\delta_{ij}-\frac{\xi_i\xi_j}{|\xi|^2})_{i,j=1}^3$, where
$\delta_{ij}=1$ if $i= j$, and $\delta_{ij}=0$ if $i\neq j$.  Recall that we have $\mathbb{P}(\nabla\pi)=0$, and if a vector field $u$ is divergence free then we have the identity $\mathbb{P}(u)=u$. Thus we rewrite the Navier--Stokes equations \eqref{eq1.1} as
 \begin{equation} \label{eq4.1}
\left\{
\begin{aligned}
& \partial_t u -\Delta u +\mathbb{P} \nabla\cdot(u\otimes u)=\mathbb{P}f,  \\
& \nabla\cdot u=0,\\
& u(0,x) =u_0(x).
\end{aligned}
\right.
\end{equation}
Based on the framework of the Kato's analytical semigroup, we can further rewrite the system \eqref{eq4.1} as an equivalent integral form:
\begin{equation}\label{eq4.2}
  u(t,x)=e^{t\Delta}u_{0}(x)+\int_{0}^{t}e^{(t-\tau)\Delta}\mathbb{P}fd\tau-\int_{0}^{t}e^{(t-\tau)\Delta}\mathbb{P}\nabla\cdot(u\otimes u)d\tau.
\end{equation}
For simplicity, we denote the third term on the right-hand side of \eqref{eq4.2}  as
\begin{equation*}
  \mathcal{B}_{1}(u,u):=\int_{0}^{t}e^{(t-\tau)\Delta}\mathbb{P}\nabla\cdot(u\otimes u)d\tau.
\end{equation*}

Now,  let the assumptions of Theorem \ref{th1.1} be in force, we introduce the solution space $\mathcal{X}_{t}$ as
$$
\mathcal{X}_{t}:=\mathcal{L}^{\rho}(0,\infty; F\dot{B}^{2-\frac{3}{p(\cdot)}+\frac{2}{\rho}}_{p(\cdot),1}(\mathbb{R}^{3}))\cap \mathcal{L}^{1}(0,\infty; F\dot{B}^{\frac{5}{2}}_{2,1}(\mathbb{R}^{3}))\cap \mathcal{L}^{\infty}(0,\infty; F\dot{B}^{\frac{1}{2}}_{2,1}(\mathbb{R}^{3}))
$$
and the norm of the  space $\mathcal{X}_{t}$ is endowed by
\begin{equation*}
   \|u\|_{\mathcal{X}_{t}}:=\max\left\{\|u\|_{\mathcal{L}^{\rho}_{t}( F\dot{B}^{2-\frac{3}{p(\cdot)}+\frac{2}{\rho}}_{p(\cdot),1})}, \|u\|_{\mathcal{L}^{1}_{t}( F\dot{B}^{\frac{5}{2}}_{2,1})}, \|u\|_{\mathcal{L}^{\infty}_{t}(F\dot{B}^{\frac{1}{2}}_{2,1})} \right\}.
\end{equation*}
It can be easily seen that $(\mathcal{X}_{t},  \|\cdot\|_{\mathcal{X}_{t}})$ is a Banach space.
\smallbreak

In the sequel, we shall establish the linear and nonlinear estimates of the integral system \eqref{eq4.2} in the space $\mathcal{X}_{t}$, respectively.
\begin{lemma}\label{le4.1}
Let the assumptions of Theorem \ref{th1.1} be in force. Then for any $u_{0}\in F\dot{B}^{2-\frac{3}{p(\cdot)}}_{p(\cdot),1}(\mathbb{R}^{3})$,
there exists a positive constant $C$ such that
\begin{equation}\label{eq4.3}
   \left\|e^{t\Delta} u_{0}\right\|_{\mathcal{X}_{t}}\leq
   C\left\|u_{0}\right\|_{F\dot{B}^{2-\frac{3}{p(\cdot)}}_{p(\cdot),1}}.
\end{equation}
\end{lemma}
\begin{proof}
We split the proof into the following two steps.

\textbf{Step 1}.  Applying Theorem \ref{th3.1}, one can easily see that
\begin{equation}\label{eq4.4}
\|e^{t\Delta} u_{0}\|_{\mathcal{L}^{\rho}_{t}(F\dot{B}^{2-\frac{3}{p(\cdot)}+\frac{2}{\rho}}_{p(\cdot),1})}\lesssim \|u_{0}\|_{ F\dot{B}^{2-\frac{3}{p(\cdot)}}_{p(\cdot),1}}.
\end{equation}

\textbf{Step 2}.  For any $1\leq \rho \leq \infty$, we can derive that
\begin{align}\label{eq4.5}
   \left\|e^{t\Delta} u_{0}\right\|_{\mathcal{L}^{\rho}_{t}(F\dot{B}^{\frac{1}{2}+\frac{2}{\rho}}_{2,1})}
   &= \sum_{j\in \mathbb{Z}}\left\|2^{j(\frac{1}{2}+\frac{2}{\rho})}\varphi_{j}e^{-t|\cdot|^{2}}\widehat{u}_{0}\right\|_{L^{\rho}_{t}(L^{2}_{\xi})}\nonumber\\
   &  \lesssim\sum_{j\in \mathbb{Z}} \sum_{k=-1}^{1}\left\|2^{j(2-\frac{3}{p(\cdot)})}\varphi_{j}\widehat{u}_{0}\right\|_{L^{p(\cdot)}_{\xi}}
   \left\|2^{j(\frac{2}{\rho}-\frac{3}{2}+\frac{3}{p(\cdot)})}\varphi_{j+k}e^{-t|\cdot|^{2}}\right\|_{L^{\rho}_{t}(L^{\frac{2p(\cdot)}{p(\cdot)-2}}_{\xi})}.
\end{align}
Notice that
\begin{align}\label{eq4.6}
   \left\|2^{j(\frac{2}{\rho}-\frac{3}{2}+\frac{3}{p(\cdot)})}\varphi_{j+k}e^{-t|\cdot|^{2}}\right\|_{L^{\rho}_{t}(L^{\frac{2p(\cdot)}{p(\cdot)-2}}_{\xi})}
   &\lesssim  \left\|2^{j\frac{2}{\rho}}e^{-t2^{2j}}\right\|_{L^{\rho}_{t}} \left\|2^{j(-\frac{3}{2}+\frac{3}{p(\cdot)})}\varphi_{j}\right\|_{L^{\frac{2p(\cdot)}{p(\cdot)-2}}_{\xi}}\nonumber\\
   &\lesssim\left\|2^{j(-\frac{3}{2}+\frac{3}{p(\cdot)})}\varphi_{j}\right\|_{L^{\frac{2p(\cdot)}{p(\cdot)-2}}_{\xi}},
\end{align}
and
\begin{align}\label{eq4.7}
   \left\|2^{j(-\frac{3}{2}+\frac{3}{p(\cdot)})}\varphi_{j}\right\|_{L^{\frac{2p(\cdot)}{p(\cdot)-2}}_{\xi}}&= \left\{\lambda>0:\ \int_{\mathbb{R}^{3}}\left|\frac{2^{j(-\frac{3}{2}+\frac{3}{p(\xi)})}\varphi_{j}(\xi)}{\lambda}\right|^{\frac{2p(\xi)}{p(\xi)-2}}d\xi<1\right\}\nonumber\\
   &=\left\{\lambda>0:\ \int_{\mathbb{R}^{3}}\left|\frac{\varphi_{j}(\xi)}{\lambda}\right|^{\frac{2p(\xi)}{p(\xi)-2}}2^{-3j}d\xi<1\right\}\nonumber\\
   &\lesssim \left\{\lambda>0:\ \int_{\mathbb{R}^{3}}\left|\frac{\varphi_{j}(2^{j}\xi)}{\lambda}\right|^{\frac{2p(2^{j}\xi)}{p(2^{j}\xi)-2}}d\xi<1\right\}\nonumber\\
   &\lesssim 1.
\end{align}
Taking the above two estimates \eqref{eq4.6} and \eqref{eq4.7} into
\eqref{eq4.5}, we obtain
\begin{align*}
   \left\|e^{t\Delta} u_{0}\right\|_{\mathcal{L}^{\rho}_{t}(\dot{B}^{\frac{1}{2}+\frac{2}{\rho}}_{2,1})}
      \lesssim\sum_{j\in \mathbb{Z}}\left\|2^{j(2-\frac{3}{p(\cdot)})}\varphi_{j}\widehat{u}_{0}\right\|_{L^{p(\cdot)}_{\xi}}
   \lesssim\|u_{0}\|_{F\dot{B}^{2-\frac{3}{p(\cdot)}}_{p(\cdot),1}}.
\end{align*}
Taking $\rho=1$ and $\rho=\infty$ in above inequality, respectively, we get
\begin{align}\label{eq4.8}
   \left\|e^{t\Delta} u_{0}\right\|_{\mathcal{L}^{1}_{t}(\dot{B}^{\frac{5}{2}}_{2,1})}+ \left\|e^{t\Delta} u_{0}\right\|_{\mathcal{L}^{\infty}_{t}(\dot{B}^{\frac{1}{2}}_{2,1})}
         \lesssim\|u_{0}\|_{F\dot{B}^{2-\frac{3}{p(\cdot)}}_{p(\cdot),1}}.
\end{align}
We have completed the proof of Lemma \ref{le4.1}.
\end{proof}
\smallbreak

\begin{lemma}\label{le4.2}
Let the assumptions of Theorem \ref{th1.1} be in force. Then for any $f\in\mathcal{L}^{1}(0,\infty; F\dot{B}^{2-\frac{3}{p(\cdot)}}_{p(\cdot),1}(\mathbb{R}^{3}))\cap \mathcal{L}^{1}(0,\infty; F\dot{B}^{\frac{1}{2}}_{2,1}(\mathbb{R}^{3}))$, there exists a positive constant $C$ such that
\begin{equation}\label{eq4.9}
   \left\|\int_{0}^{t}e^{(t-\tau)\Delta}\mathbb{P}fd\tau\right\|_{\mathcal{X}_{t}}\leq
   C\left(\|f\|_{\mathcal{L}^{1}_{t}(F\dot{B}^{2-\frac{3}{p(\cdot)}}_{p(\cdot),1})}+\|f\|_{\mathcal{L}^{1}_{t}(F\dot{B}^{\frac{1}{2}}_{2,1})}\right).
\end{equation}
\end{lemma}
\begin{proof}
Since the Leray projector $\mathbb{P}$ is bounded under the Fourier variables, we can exactly  obtain \eqref{eq4.9} by Theorem \ref{th3.1}.
\end{proof}

\begin{lemma}\label{le4.3}
Let the assumptions of Theorem \ref{th1.1} be in force.  Then
there exists a positive constant $C$ such that
\begin{equation}\label{eq4.10}
   \left\|\mathcal{B}_{1}(u,u)\right\|_{\mathcal{X}_{t}}\leq C
    \|u\|_{\mathcal{X}_{t}}^{2}.
\end{equation}
\end{lemma}
\begin{proof}
We split the proof into the following two steps.

\textbf{Step 1}.  According to Theorem \ref{th3.1}, it is clear that
\begin{align}\label{eq4.11}
   \left\|\mathcal{B}_{1}(u,u)\right\|_{\mathcal{L}^{\rho}_{t}(F\dot{B}^{2-\frac{3}{p(\cdot)}+\frac{2}{\rho}}_{p(\cdot),1})}&\lesssim \left\|\nabla\cdot(u\otimes u)\right\|_{\mathcal{L}^{1}_{t}(F\dot{B}^{2-\frac{3}{p(\cdot)}}_{p(\cdot),1})}\nonumber\\&\lesssim \left\|u\otimes u\right\|_{\mathcal{L}^{1}_{t}(F\dot{B}^{3-\frac{3}{p(\cdot)}}_{p(\cdot),1})}.
\end{align}
Furthermore, we can estimate the right-hand side of \eqref{eq4.11} as
\begin{align}\label{eq4.12}
   \left\|u\otimes u\right\|_{\mathcal{L}^{1}_{t}(F\dot{B}^{3-\frac{3}{p(\cdot)}}_{p(\cdot),1})}
   &=\sum_{j\in\mathbb{Z}}\left\|2^{j(3-\frac{3}{p(\cdot)})}\varphi_{j}\mathcal{F}\left(u\otimes u\right)\right\|_{L^{1}_{t}(L^{p(\cdot)}_{\xi})}\nonumber\\
   &\lesssim\sum_{j\in\mathbb{Z}}\sum_{k=-1}^{1}\left\|2^{-3j(\frac{6-p(\cdot)}{6p(\cdot)})}\varphi_{j}\right\|_{L^{\frac{6p(\cdot)}{6-p(\cdot)}}_{\xi}}
   \left\|2^{\frac{5}{2}j}\varphi_{j+k}\mathcal{F}\left(u\otimes u\right)\right\|_{L^{1}_{t}L^{6}_{\xi}}\nonumber\\
  \nonumber\\
   &\lesssim \left\|2^{\frac{5}{2}j}\varphi_{j+k}\mathcal{F}\left(u\otimes u\right)\right\|_{L^{1}_{t}L^{6}_{\xi}}\nonumber\\
   &\approx  \left\|u\otimes u\right\|_{\mathcal{L}^{1}_{t}(F\dot{B}^{\frac{5}{2}}_{6,1})}.
\end{align}
By using Lemma \ref{le2.9} with $p=6$, $p_{1}=2$ and $p_{2}=\frac{3}{2}$, we obtain
\begin{align}\label{eq4.13}
 \left\|u\otimes u\right\|_{\mathcal{L}^{1}_{t}(F\dot{B}^{\frac{5}{2}}_{6,1})}
 &\lesssim  \left\|u\right\|_{\mathcal{L}^{1}_{t}(F\dot{B}^{\frac{5}{2}}_{2,1})}\left\|u\right\|_{\mathcal{L}^{\infty}_{t}(F\dot{B}^{0}_{\frac{3}{2},1})}\nonumber\\
 &\lesssim  \left\|u\right\|_{\mathcal{L}^{1}_{t}(F\dot{B}^{\frac{5}{2}}_{2,1})}\left\|u\right\|_{\mathcal{L}^{\infty}_{t}(F\dot{B}^{\frac{1}{2}}_{2,1})},
\end{align}
where we have used the Sobolev embedding $F\dot{B}^{\frac{1}{2}}_{2,1}(\mathbb{R}^{3})\hookrightarrow F\dot{B}^{0}_{\frac{3}{2},1}(\mathbb{R}^{3})$ in Lemma \ref{le2.8}. Taking the above estimates \eqref{eq4.12} and \eqref{eq4.13} into \eqref{eq4.11}, we obtain
\begin{align}\label{eq4.14}
   \left\|\mathcal{B}_{1}(u,u)\right\|_{\mathcal{L}^{\rho}_{t}(F\dot{B}^{2-\frac{3}{p(\cdot)}+\frac{2}{\rho}}_{p(\cdot),1})}&\lesssim \left\|u\right\|_{\mathcal{L}^{1}_{t}(F\dot{B}^{\frac{5}{2}}_{2,1})}\left\|u\right\|_{\mathcal{L}^{\infty}_{t}(F\dot{B}^{\frac{1}{2}}_{2,1})}.
\end{align}

\textbf{Step 2}.  Applying Theorem \ref{th3.1} and Lemma \ref{le2.10} with $s_{1}=\frac{5}{2}$,  $s_{2}=\frac{1}{2}$ and $p_{1}=p_{2}=2$ (the indices of the time variable obey the law of H\"{o}lder's inequality), we see that
\begin{align}\label{eq4.15}
 \left\|\mathcal{B}_{1}(u,u)\right\|_{\mathcal{L}^{1}_{t}(F\dot{B}^{\frac{5}{2}}_{2,1})}+ \left\|\mathcal{B}(u,u)\right\|_{\mathcal{L}^{\infty}_{t}(F\dot{B}^{\frac{1}{2}}_{2,1})}&\leq \left\|u\otimes u\right\|_{\mathcal{L}^{1}_{t}(F\dot{B}^{\frac{3}{2}}_{2,1})}\nonumber\\
   &\leq C \left\|u\right\|_{\mathcal{L}^{1}_{t}(F\dot{B}^{\frac{5}{2}}_{2,1})}\left\|u\right\|_{\mathcal{L}^{\infty}_{t}(F\dot{B}^{\frac{1}{2}}_{2,1})}.
\end{align}
Putting the above estimates \eqref{eq4.14} and \eqref{eq4.15} together, we get \eqref{eq4.10}.  We have completed the proof of Lemma \ref{le4.3}.
\end{proof}
\smallbreak

Based on the desired linear and nonlinear estimates obtained in Lemmas \ref{le4.1}--\ref{le4.3}, we know that there exist two positive
constants $C_{1}$ and $C_{2}$ such that
\begin{align*}
  \|u\|_{\mathcal{X}_{t}}\leq
 C_{1}\left(\|u_{0}\|_{F\dot{B}^{2-\frac{3}{p(\cdot)}}_{p(\cdot),1}}+\|f\|_{\mathcal{L}^{1}_{t}(F\dot{B}^{2-\frac{3}{p(\cdot)}}_{p(\cdot),1})}
 +\|f\|_{\mathcal{L}^{1}_{t}(F\dot{B}^{\frac{1}{2}}_{2,1})}\right)+C_{2}\|u\|_{\mathcal{X}_{t}}^2.
\end{align*}
Thus if the initial data $u_{0}$  and $f$ satisfy the condition
$$
\|u_{0}\|_{F\dot{B}^{2-\frac{3}{p(\cdot)}}_{p(\cdot),1}}+\|f\|_{\mathcal{L}^{1}_{t}(F\dot{B}^{2-\frac{3}{p(\cdot)}}_{p(\cdot),1})}
+\|f\|_{\mathcal{L}^{1}_{t}(F\dot{B}^{\frac{1}{2}}_{2,1})}\leq \frac{1}{4C_{1}C_{2}},
$$
we can apply Proposition \ref{pro2.13} to get global existence of solution $u\in \mathcal{X}_{t}$ to the equations \eqref{eq1.1}. We have completed the proof of Theorem \ref{th1.1}.

\section{Proof of Theorem \ref{th1.2}}

Notice that the chemical concentration $v$ is determined by the Poisson equation, i.e., the second equation of \eqref{eq1.2}, gives rise to the coefficient $\nabla v$ in the first equation of \eqref{eq1.2},  thus $v$ can be represented as the volume potential of $u$:
\begin{equation*}
v(t,x)=(-\Delta)^{-1}u(t,x)=
\frac{1}{3\omega_{3}}\int_{\mathbb{R}^{3}}\frac{u(t,y)}{|x-y|}dy,
\end{equation*}
where $\omega_{3}$ denotes the surface area of the unit sphere in
$\mathbb{R}^{3}$. Based on the framework of the Kato's analytical semigroup, we can rewrite system \eqref{eq1.2} as an equivalent integral form:
\begin{equation}\label{eq5.1}
  u(t,x)=e^{t\Delta}u_{0}(x)+\int_{0}^{t}e^{(t-\tau)\Delta}fd\tau+\int_{0}^{t}e^{(t-\tau)\Delta}\nabla\cdot(u\nabla(-\Delta)^{-1}u)d\tau.
\end{equation}
For simplicity, we denote the third term on the right-hand side of \eqref{eq5.1}  as
\begin{equation*}
  \mathcal{B}_{2}(u,u):=\int_{0}^{t}e^{(t-\tau)\Delta}\nabla\cdot(u\nabla(-\Delta)^{-1}u)d\tau.
\end{equation*}

Now,  let the assumptions of Theorem \ref{th1.2} be in force, we introduce the solution space $\mathcal{Y}_{t}$ as follows:
$$
\mathcal{Y}_{t}:=\mathcal{L}^{\rho}(0,\infty; F\dot{B}^{1-\frac{3}{p(\cdot)}+\frac{2}{\rho}}_{p(\cdot),1}(\mathbb{R}^{3}))\cap \mathcal{L}^{1}(0,\infty; F\dot{B}^{\frac{3}{2}}_{2,1}(\mathbb{R}^{3}))\cap \mathcal{L}^{\infty}(0,\infty; F\dot{B}^{-\frac{1}{2}}_{2,1}(\mathbb{R}^{3}))
$$
and the norm of the  space $\mathcal{Y}_{t}$ is endowed by
\begin{equation*}
   \|u\|_{\mathcal{Y}_{t}}=\max\left\{\|u\|_{\mathcal{L}^{\rho}_{t}( F\dot{B}^{1-\frac{3}{p(\cdot)}+\frac{2}{\rho}}_{p(\cdot),1})}, \|u\|_{\mathcal{L}^{1}_{t}( F\dot{B}^{\frac{3}{2}}_{2,1})}, \|u\|_{\mathcal{L}^{\infty}_{t}(F\dot{B}^{-\frac{1}{2}}_{2,1})} \right\}.
\end{equation*}
It can be easily seen that $(\mathcal{Y}_{t},  \|\cdot\|_{\mathcal{Y}_{t}})$ is a Banach space.
\smallbreak

In the sequel, we shall establish the linear and nonlinear estimates of the integral equation \eqref{eq5.1} in the space $\mathcal{Y}_{t}$, respectively.
\begin{lemma}\label{le5.1}
Let the assumptions of Theorem \ref{th1.2} be in force. Then for any $u_{0}\in F\dot{B}^{1-\frac{3}{p(\cdot)}}_{p(\cdot),1}(\mathbb{R}^{3})$,
there exists a positive constant $C$ such that
\begin{equation}\label{eq5.2}
   \|e^{t\Delta} u_{0}\|_{\mathcal{Y}_{t}}\leq
   C\|u_{0}\|_{F\dot{B}^{1-\frac{3}{p(\cdot)}}_{p(\cdot),1}}.
\end{equation}
\end{lemma}
\begin{proof}
We split the proof into the following two steps.

\textbf{Step 1}.  Applying Theorem \ref{th3.1}, one can easily see that
\begin{equation}\label{eq5.3}
\|e^{t\Delta} u_{0}\|_{\mathcal{L}^{\rho}_{t}(F\dot{B}^{1-\frac{3}{p(\cdot)}+\frac{2}{\rho}}_{p(\cdot),1})}\leq C\|u_{0}\|_{ F\dot{B}^{1-\frac{3}{p(\cdot)}}_{p(\cdot),1}}.
\end{equation}

\textbf{Step 2}. For any $1\leq \rho \leq \infty$, we can use the facts \eqref{eq4.6} and \eqref{eq4.7} to derive that
\begin{align}\label{eq5.4}
   \left\|e^{t\Delta} u_{0}\right\|_{\mathcal{L}^{\rho}_{t}(F\dot{B}^{-\frac{1}{2}+\frac{2}{\rho}}_{2,1})}
   &=\sum_{j\in \mathbb{Z}}\left\|2^{j(-\frac{1}{2}+\frac{2}{\rho})}\varphi_{j}e^{-t|\cdot|^{2}}\widehat{u}_{0}\right\|_{L^{\rho}_{t}(L^{2}_{\xi})}\nonumber\\
   &\lesssim\sum_{j\in \mathbb{Z}} \sum_{k=-1}^{1}\left\|2^{j(1-\frac{3}{p(\cdot)})}\varphi_{j}\widehat{u}_{0}\right\|_{L^{p(\cdot)}_{\xi}}
   \left\|2^{j(\frac{2}{\rho}-\frac{3}{2}+\frac{3}{p(\cdot)})}\varphi_{j+k}e^{-t|\cdot|^{2}}\right\|_{L^{\rho}_{t}(L^{\frac{2p(\cdot)}{p(\cdot)-2}}_{\xi})}\nonumber\\
    &\lesssim\sum_{j\in \mathbb{Z}}\left\|2^{j(1-\frac{3}{p(\cdot)})}\varphi_{j}\widehat{u}_{0}\right\|_{L^{p(\cdot)}_{\xi}}\nonumber\\
   &\lesssim\|u_{0}\|_{F\dot{B}^{1-\frac{3}{p(\cdot)}}_{p(\cdot),1}}.
\end{align}
Taking $\rho=1$ and $\rho=\infty$ in above inequality \eqref{eq5.4}, respectively, we get
\begin{align}\label{eq5.5}
   \left\|e^{t\Delta} u_{0}\right\|_{\mathcal{L}^{1}_{t}(\dot{B}^{\frac{3}{2}}_{2,1})}+ \left\|e^{t\Delta} u_{0}\right\|_{\mathcal{L}^{\infty}_{t}(\dot{B}^{-\frac{1}{2}}_{2,1})}
         \lesssim\|u_{0}\|_{F\dot{B}^{1-\frac{3}{p(\cdot)}}_{p(\cdot),1}}.
\end{align}
We have completed the proof of Lemma \ref{le5.1}.
\end{proof}

\begin{lemma}\label{le5.2}
Let the assumptions of Theorem \ref{th1.2} be in force. Then for any $f\in\mathcal{L}^{1}(0,\infty; F\dot{B}^{1-\frac{3}{p(\cdot)}}_{p(\cdot),1}(\mathbb{R}^{3}))\cap \mathcal{L}^{1}(0,\infty; F\dot{B}^{-\frac{1}{2}}_{2,1}(\mathbb{R}^{3}))$, there exists a positive constant $C$ such that
\begin{equation}\label{eq5.6}
   \left\|\int_{0}^{t}e^{(t-\tau)\Delta}fd\tau\right\|_{\mathcal{Y}_{t}}\leq
   C\left(\|f\|_{\mathcal{L}^{1}_{t}(F\dot{B}^{1-\frac{3}{p(\cdot)}}_{p(\cdot),1})}+\|f\|_{\mathcal{L}^{1}_{t}(F\dot{B}^{-\frac{1}{2}}_{2,1})}\right).
\end{equation}
\end{lemma}
\begin{proof}
This can be exactly obtained by Theorem \ref{th3.1}.
\end{proof}

\begin{lemma}\label{le5.3}
Let the assumptions of Theorem \ref{th1.2} be in force.  Then
there exists a positive constant $C$ such that
\begin{equation}\label{eq5.7}
   \left\|\mathcal{B}_{2}(u,u)\right\|_{\mathcal{Y}_{t}}\leq C
    \|u\|_{\mathcal{Y}_{t}}^{2}.
\end{equation}
\end{lemma}
\begin{proof}We split the proof into the following two steps.

\textbf{Step 1}.  According to Theorem \ref{th3.1} and Lemma \ref{le2.8}, it is clear that
\begin{align}\label{eq5.8}
   \left\|\mathcal{B}_{2}(u,u)\right\|_{\mathcal{L}^{\rho}_{t}(F\dot{B}^{1-\frac{3}{p(\cdot)}+\frac{2}{\rho}}_{p(\cdot),1})}&\lesssim \left\|\nabla\cdot(u\nabla(-\Delta)^{-1}u)\right\|_{\mathcal{L}^{1}_{t}(F\dot{B}^{1-\frac{3}{p(\cdot)}}_{p(\cdot),1})}\nonumber\\&\approx \left\|u\nabla(-\Delta)^{-1}u\right\|_{\mathcal{L}^{1}_{t}(F\dot{B}^{2-\frac{3}{p(\cdot)}}_{p(\cdot),1})}.
\end{align}
Furthermore, we can estimate the right-hand side of \eqref{eq5.8} as
\begin{align}\label{eq5.9}
   \left\|u\nabla(-\Delta)^{-1}u\right\|_{\mathcal{L}^{1}_{t}(F\dot{B}^{2-\frac{3}{p(\cdot)}}_{p(\cdot),1})}
   &=\sum_{j\in\mathbb{Z}}\left\|2^{j(2-\frac{3}{p(\cdot)})}\varphi_{j}\mathcal{F}\left(u\nabla(-\Delta)^{-1}u\right)\right\|_{L^{1}_{t}(L^{p(\cdot)}_{\xi})}\nonumber\\
   &\lesssim\sum_{j\in\mathbb{Z}}\sum_{k=-1}^{1}\left\|2^{-3j(\frac{6-p(\cdot)}{6p(\cdot)})}\varphi_{j}\right\|_{L^{\frac{6p(\cdot)}{6-p(\cdot)}}_{\xi}}
   \left\|2^{\frac{3}{2}j}\varphi_{j+k}\mathcal{F}\left(u\nabla(-\Delta)^{-1}u\right)\right\|_{L^{1}_{t}L^{6}_{\xi}}\nonumber\\
  \nonumber\\
   &\lesssim \left\|2^{\frac{3}{2}j}\varphi_{j+k}\mathcal{F}\left(u\nabla(-\Delta)^{-1}u\right)\right\|_{L^{1}_{t}L^{6}_{\xi}}\nonumber\\
   &\approx  \left\|u\nabla(-\Delta)^{-1}u\right\|_{\mathcal{L}^{1}_{t}(F\dot{B}^{\frac{3}{2}}_{6,1})}.
\end{align}
Notice carefully that the nonlinear term $u\nabla(-\Delta)^{-1}u$ has a nice symmetric structure:
\begin{equation*}
    u\nabla(-\Delta)^{-1}u=-\nabla\cdot\left(\nabla(-\Delta)^{-1}u\otimes\nabla(-\Delta)^{-1}u-\frac{1}{2}\left|\nabla(-\Delta)^{-1}u\right|^{2}I\right),
\end{equation*}
where $\otimes$ is a tensor product, and $I$ is a $3$-th order identity matrix, thus by using Lemma \ref{le2.9} with $p=6$, $p_{1}=2$ and $p_{2}=\frac{3}{2}$, we obtain
\begin{align}\label{eq5.10}
 \left\|u\nabla(-\Delta)^{-1}u\right\|_{\mathcal{L}^{1}_{t}(F\dot{B}^{\frac{3}{2}}_{6,1})}
 &\approx\left\|\nabla(-\Delta)^{-1}u\otimes\nabla(-\Delta)^{-1}u-\frac{1}{2}\left|\nabla(-\Delta)^{-1}u\right|^{2}I\right\|_{\mathcal{L}^{1}_{t}(F\dot{B}^{\frac{5}{2}}_{6,1})}\nonumber\\
 &\lesssim  \left\|\nabla(-\Delta)^{-1}u\right\|_{\mathcal{L}^{1}_{t}(F\dot{B}^{\frac{5}{2}}_{2,1})}\left\|\nabla(-\Delta)^{-1}u\right\|_{\mathcal{L}^{\infty}_{t}(F\dot{B}^{0}_{\frac{3}{2},1})}\nonumber\\
 &\lesssim  \left\|\nabla(-\Delta)^{-1}u\right\|_{\mathcal{L}^{1}_{t}(F\dot{B}^{\frac{5}{2}}_{2,1})}\left\|\nabla(-\Delta)^{-1}u\right\|_{\mathcal{L}^{\infty}_{t}(F\dot{B}^{\frac{1}{2}}_{2,1})}\nonumber\\
  &\lesssim  \left\|u\right\|_{\mathcal{L}^{1}_{t}(F\dot{B}^{\frac{3}{2}}_{2,1})}\left\|u\right\|_{\mathcal{L}^{\infty}_{t}(F\dot{B}^{-\frac{1}{2}}_{2,1})},
\end{align}
where we have used the Sobolev embedding $F\dot{B}^{\frac{1}{2}}_{2,1}(\mathbb{R}^{3})\hookrightarrow F\dot{B}^{0}_{\frac{3}{2},1}(\mathbb{R}^{3})$ in Lemma \ref{le2.8}. Taking the above estimates \eqref{eq5.9} and \eqref{eq5.10} into \eqref{eq5.8}, we obtain
\begin{align}\label{eq5.11}
   \left\|\mathcal{B}_{2}(u,u)\right\|_{\mathcal{L}^{\rho}_{t}(F\dot{B}^{2-\frac{3}{p(\cdot)}+\frac{2}{\rho}}_{p(\cdot),1})}&\lesssim \left\|u\right\|_{\mathcal{L}^{1}_{t}(F\dot{B}^{\frac{3}{2}}_{2,1})}\left\|u\right\|_{\mathcal{L}^{\infty}_{t}(F\dot{B}^{\frac{1}{2}}_{2,1})}.
\end{align}

\textbf{Step 2}.  Applying Theorem \ref{th3.1} and Lemma \ref{le2.10} with $s_{1}=\frac{3}{2}$,  $s_{2}=\frac{1}{2}$ and $p_{1}=p_{2}=2$ (the indices of the time variable obey the law of H\"{o}lder's inequality), we obtain
\begin{align}\label{eq5.12}
 \left\|\mathcal{B}_{2}(u,u)\right\|_{\mathcal{L}^{1}_{t}(F\dot{B}^{\frac{3}{2}}_{2,1})}+\left\|\mathcal{B}_{2}(u,u)\right\|_{\mathcal{L}^{\infty}_{t}(F\dot{B}^{-\frac{1}{2}}_{2,1})}&\leq \left\|\nabla\cdot(u\nabla(-\Delta)^{-1}u)\right\|_{\mathcal{L}^{1}_{t}(F\dot{B}^{-\frac{1}{2}}_{2,1})}\nonumber\\
 &\leq \left\|u\nabla(-\Delta)^{-1}u\right\|_{\mathcal{L}^{1}_{t}(F\dot{B}^{\frac{1}{2}}_{2,1})}\nonumber\\
    &\leq C \left\|u\right\|_{\mathcal{L}^{1}_{t}(F\dot{B}^{\frac{3}{2}}_{2,1})}\left\|\nabla(-\Delta)^{-1}u\right\|_{\mathcal{L}^{\infty}_{t}(F\dot{B}^{\frac{1}{2}}_{2,1})}
    \nonumber\\
    &\leq C \left\|u\right\|_{\mathcal{L}^{1}_{t}(F\dot{B}^{\frac{3}{2}}_{2,1})}\left\|u\right\|_{\mathcal{L}^{\infty}_{t}(F\dot{B}^{-\frac{1}{2}}_{2,1})}.
\end{align}
Putting the above estimates \eqref{eq5.11} and \eqref{eq5.12} together, we get \eqref{eq5.7}.  We have completed the proof of Lemma \ref{le5.3}.
\end{proof}
\smallbreak

Based on the desired linear and nonlinear estimates obtained in Lemmas \ref{le5.1}--\ref{le5.3}, we know that there exist two positive
constants $C_{3}$ and $C_{4}$ such that
\begin{align*}
  \|u\|_{\mathcal{X}_{t}}\leq
 C_{3}\left(\|u_{0}\|_{F\dot{B}^{1-\frac{3}{p(\cdot)}}_{p(\cdot),1}}+\|f\|_{\mathcal{L}^{1}_{t}(F\dot{B}^{1-\frac{3}{p(\cdot)}}_{p(\cdot),1})}
 +\|f\|_{\mathcal{L}^{1}_{t}(F\dot{B}^{-\frac{1}{2}}_{2,1})}\right)+C_{4}\|u\|_{\mathcal{X}_{t}}^2.
\end{align*}
Thus if the initial data $u_{0}$ and $f$ satisfy the condition
$$
\|u_{0}\|_{F\dot{B}^{1-\frac{3}{p(\cdot)}}_{p(\cdot),1}}+\|f\|_{\mathcal{L}^{1}_{t}(F\dot{B}^{1-\frac{3}{p(\cdot)}}_{p(\cdot),1})}
+\|f\|_{\mathcal{L}^{1}_{t}(F\dot{B}^{-\frac{1}{2}}_{2,1})}\leq \frac{1}{4C_{3}C_{4}},
$$
we can apply Proposition \ref{pro2.13} to get global existence of solution $u\in \mathcal{Y}_{t}$ to the system \eqref{eq1.2}. We have completed the proof of Theorem \ref{th1.2}.
\bigskip

%%%%%%%%%%%%%%%%%%%%%%%%%%%%%%%%%%%%%%%%%%%%%%%%%%%%%%
\noindent \textbf{Acknowledgements.}
The authors declared that they have no conflict of interest. G. Vergara-Hermosilla thanks to Diego Chamorro for their helpful comments and advises.
J. Zhao was partially supported by the National Natural Science Foundation of China (no. 12361034) and
 the Natural Science Foundation of Shaanxi Province (no. 2022JM-034).

%%%%%%%%%%%%%%%%%%%%%%%%%%%%%%%%%%%%%%%%%%%%%%%%%%%	

\end{document}